\documentclass[12pt, a4paper]{article}
\usepackage{ucs}
\usepackage[sumlimits, namelimits]{amsmath}
\usepackage{longtable,multicol,verbatim,graphics,graphicx,lscape,animate}

\usepackage{url}
\usepackage{graphicx}
\usepackage{amsmath, amsfonts, amsthm, amssymb}
\newtheorem{theorem}{Theorem}[subsection]
\newtheorem{prove}{Proof of Theorem}[subsection]
\newtheorem{lemma}[theorem]{Lemma}
\newtheorem{definition}{Definition}[subsection]
\newtheorem{proposition}[definition]{Proposition}

\begin{document}
\title{\bf{CLASSIFICATION OF 5-DIMENSIONAL MD-ALGEBRAS HAVING NON-COMMUTATIVE DERIVED IDEALS}}
\author{{\bf Le Anh} ${\bf Vu}^{*}$, {\bf Ha Van} ${\bf Hieu}^{**}$ and 
{\bf Tran Thi Hieu} ${\bf Nghia}^{***}$\\
${}^{*}$\footnotesize{Department of Mathematics and Economic Statistics, University of Economics and Law}\\
\footnotesize{Vietnam National University - Ho Chi Minh City, Viet nam}\\
\footnotesize{E-mail:\, vula@uel.edu.vn}\\
${}^{**}$ \footnotesize{E-mail:\, havanhieu88@gmail.com}\\
${}^{***}$ \footnotesize{E-mail:\, hieunghiatoan1a@gmail.com}}
\date{}
\maketitle
\begin{abstract}
The paper presents a subclass of the class of MD5-algebras and MD5-groups, i.e. five dimensional solvable Lie algebras and Lie groups such that their orbits in the co-adjoint representation (K-orbits) are orbits of zero or maximal dimension. The main result of the paper is the classification up to an isomorphism of all MD5-algebras having non-commutative derived ideals.
\end{abstract}


\subsection*{INTRODUCTION}
In 1962, studying theory of representations, A. A. Kirillov introduced the Orbit Method (see [2]). This method quickly became the most important method in the theory of representations of Lie groups. Using the Kirillov's Orbit Method, we can obtain all the unitary irreducible representations of solvable and simply connected Lie Groups. The importance of Kirillov's Orbit Method is the co-adjoint representation (K-representation). Therefore, it is meaningful to study the K-representation in the theory of representations of Lie groups.

After studying the Kirillov's Orbit Method, Do Ngoc Diep in 1980 suggested to consider the class of Lie groups and Lie Algebras MD such that the $C^*-algebras$ of them can be described by using KK-functors (see [1]). Let G be an n-dimensional real Lie group. G is called an MDn-group if and only if its orbits in the K-representation (i.e. K-orbits) are orbits of dimension zero or maximal dimension. The corresponding Lie algebra of G is called an MDn-algebra. Thus, classification and studying of K-representation of the class of MDn-groups and MDn-algebras is the problem of interest. Because all Lie algebras of n dimension (with $n \le 3$) were listed easily, we have to consider MDn-groups and MDn-algebras with $n \ge 4$. In 1990, all MD4-algebras were classified up to an isomorphism by Vu - the first author (see [5]). Recently, Vu and some his colleagues have continued studying MD5-groups and MD5-algebras having commutative derived ideals (see [6], [7], [8]). In 2008, a classification of all MD5-algebras having commutative derived ideals was given by Vu and Kar Ping Shum (see [9]).

In this paper, we shall give the classification up to an isomorphism of all MD5-algebras $\mathcal{G}$ whose derived ideals ${\mathcal{G}}^{1}:=[\mathcal{G}, \mathcal{G}]$ are non-commutative. This classification is the main result of the paper.

The paper is organized as follows: The first section deals with some preliminary notions, section 2 is devoted to the discussion of some results on MDn-algebras, in particular, the main result of the paper is given in this section.
%-----------------------------------------------------------------------
\section{PRELIMINARIES}
We first recall in this section some preliminary results and notations which will be used later. For details we refer the reader to the book [2] of A. A. Kirillov and the book [1] of Do Ngoc Diep.

\subsection{The K-representation and K-orbits}
Let G be a Lie group, $\cal G$ =Lie(G)  be the corresponding Lie algebra of G and ${\cal G}^*$ be the dual space of $\cal G$. For every $g \in G$, we denote the internal automorphism associated with g by $A_{(g)}$, and whence,   $A_{(g)}:G \to G$ can be defined as follows $A_{(g)}:=g.x.g^{-1},\forall x \in G$. This automorphism induces the following map ${A_{\left(g \right)}}_{*} : {\cal G} \to {\cal G}$  which is defined as follows 

\centerline{${A_{\left( g \right)}}_{*} \left( X \right): = \frac{d}{{dt}}\left[ {g.\exp \left( {tX} \right).g^{ - 1} } \right]\left| {_{t = 0} } \right.$,\, $\forall X \in {\cal G}.$}

This map is called \textit{the tangent map} of  $A_{(g)}$.
We now formulate the definitions of K-representation and K-orbit.
%-------------------------------------------Definition 1.1.1
\begin{definition}The action
$$K:G\longrightarrow Aut(\mathcal{G}^{*})$$
$$g\longmapsto K_{(g)}$$
such that
$$\left\langle {K_{\left( g \right)} (F), X} \right\rangle : = \left\langle {F, {A_{\left(g^{ - 1} \right)}}_{*}\left(X \right)} \right\rangle, \forall F \in{\cal G}^* ,\, \forall X \in {\cal G}$$ 
is called the co-adjoint representation or K-representation of G in $\mathcal{{G}^{*}}$.\\
\end{definition}
%------------------------Definition 1.1.2 
\begin{definition} Each orbit of the co-adjoint representation
of G is called a K-orbit of G.
\end{definition}
We denote the K-orbit containing F by $\Omega_F$. Thus, for every $F \in {\cal G}^*$, we have $\Omega_F  := \left\{ {K\left( g \right)(F)|g \in G}\right\}$. The dimension of every K-orbit of an arbitrary Lie group G is always even. In order to define the dimension of the K-orbits $\Omega _F$ for each F from the dual space  ${\cal G}^*$  of the Lie algebra $\cal G$ = Lie(G), it is useful to consider the following skew-symmetric bilinear form $B_F$  on $\cal G$: $B_F \left( {X,Y} \right) = \left\langle {F,\left[ {X,Y} \right]} \right\rangle ,\forall X,Y \in {\cal G}$. We denote the stabilizer of F under the co-adjoint representation of G in ${\cal G}^*$  by $G_F$  and  ${\cal G}_F:=$ Lie($G_F$).

We shall need in the sequel of the following result.
%-----------------------------------Proposition 1.1.3 
\begin{proposition}[see {[2, Section 15.1]}]$KerB_{F} = {\cal{G}}_{F}$
and $dim{\Omega}_{F} = dim{\mathcal{G}} - dim{\mathcal{G}}_{F} = rankB_{F}.$ \hfill{$\square$}
\end{proposition}

\subsection{MD-groups and MD-algebras}
\begin{definition}[see {[1, Chapter 2]}] An MD-group is a real solvable Lie group of finite dimension such that its K-orbits are orbits of dimension zero or maximal dimension (i.e. dimension k, where k is some even constant and no more than the dimension of the considered group). When the dimension of considered group is n (n is a some positive integer), the group is called an MDn-group. The Lie algebra of an MD-group (MDn-group, respectively) is called an MD-algebra (MDn-algebra, respectively).
\end{definition}
The following proposition gives a necessary condition for a Lie algebra belonging to the class of MD-algebras.
%---------------------------------------------------Proposition 1.1.4
\begin{proposition}[{see [3, Theorem 4]}] Let $\mathcal{G}$ be an MD-algebra.
Then its second derived ideal ${\mathcal{G}}^{2} := [[\mathcal{G},
\mathcal{G}], [\mathcal{G}, \mathcal{G}]]$ is commutative.
\hfill{$\square$}
\end{proposition}

We point out here that the converse of the above result is in general not true. In other words, the above necessary condition is not a sufficient
condition. So, we now only consider the real solvable Lie algebras
having commutative second derived ideals. Thus, they could be MD-algebras.

\begin{proposition}[see{ [1, Chapter 2, Proposition 2.1]}] Let \,
$\cal{G}$\, be an MD-algebra with F (in ${\cal{G}}^{*}$) is
not vanishing perfectly in ${\cal{G}}^{1}: = [\mathcal{G}, \mathcal{G}]$, i.e. there exists $U \in {\cal{G}}^{1}$ such that $\langle F, U \rangle \neq 0.$ Then
the K-orbit ${\Omega}_{F}$ is one of the K-orbits having maximal dimension. \hfill{$\square$}
\end{proposition}

%--------------------------------------------Section 2 
\section{THE CLASS OF MD5-ALGEBRAS\\
HAVING NON-COMMUTATIVE\\
DERIVED IDEALS}

%---------------------- Subsection 2.1
\subsection{Some Results on the Class of MD-algebras}
In this subsection, we shall present some results on general MDn-algebras ($n \ge 4$).

Firstly, we consider a real solvable Lie algebra $\cal{G}$ of dimension $n$ such that $dim{\cal G}^{1}= n - k$\,($k$ is some integer constant, $1\leq k<n$), ${\cal G}^{2}$ is non - trivial commutative and $dim{\cal G}^{2} = dim{\cal G}^{1}-1 = n - k - 1$.
Without loss of generality, we may assume that
\begin{description}
 \item[]\hskip3cm ${\cal G}\,\,\,= gen\left(X_1, X_2,..., X_n \right),\, (n \ge 4)$,
 \item[]\hskip3cm ${{\cal G}^1} = gen\left(X_{k+1}, X_{k+2},..., X_n \right),\,
 (n > k \ge 1)$,
 \item[]\hskip3cm ${{\cal G}^2} = gen\left(X_{k + 2},..., X_n \right)$,
\end{description} 
with the Lie brackets are given by  
$$\left[X_i, X_j \right] = 
\sum\limits_{l = k+1}^n {C_{ij}^l{X_l}},\,1 \le i < j \le n,$$
where $C_{ij}^l \left({1 \le i < j \le n}, k+1 \le l \le n\right)$ are constructional constants of $\cal{G}$.
%-------------------2.1.1 
\begin{theorem} There is no MD-algebra $\cal G$ such that its second derived ideal ${\cal G}^{2}$ is not trivial and less than its first derived ideal ${\cal G}^{1}$  by one dimension: $dim{\cal G}^{2} = dim{\cal G}^{1}-1.$
\end{theorem}
%--------------- Some Lemmas
In order to prove this theorem, we need some lemmas.
%--------------------------Lemma 2.1.2
\begin{lemma}The operator $ad_{X_{k+1}}$ restricted on ${\cal G}^2$ is an automorphism.
\end{lemma}
\begin{proof}
Since ${\cal G}^2$  is commutative, $\left[ {{X_i},{X_j}} \right] = 0,\,\,\forall i,j \ge k + 2$. Hence, 
$$\begin{array}{l}
\quad\, gen\left({X_{k + 2}},{X_{k + 3}}, \cdots ,{X_n}\right)
= {{\cal G}^2} = \left[ {{{\cal G}^1},{{\cal G}^1}} \right]\\ 
= gen\left(\left[X_{k + 1},X_{k + 2} \right],\left[X_{k + 1},X_{k + 3} \right],
\cdots ,\left[X_{k + 1},X_n \right]\right)\\ 
= gen\left(ad_{X_{k + 1}}\left(X_{k + 2}\right),
\cdots, ad_{X_{k + 1}}\left(X_n \right)\right).
\end{array}$$
It follows that $ad_{X_{k+1}}$ restricted on ${\cal G}^2$ is automorphic.
\end{proof} 
%-----------------------------------------------2.1.3
\begin{lemma} Without any restriction of generality, we can always suppose right from the start that $\left[X_i,X_{k + 1} \right] = 0$ for all indices $i$ such that $1 \le i \le k$.
\end{lemma}
\begin{proof} 
Firstly, we remark that  $\left[X_1,X_{k + 1} \right] \in {\cal G}^1$, so there exists $X\in{\cal G}^2$  such that $\left[X_1,X_{k + 1} \right] 
= C_{1,k + 1}^{k + 1}X_{k + 1} + X$.
Since ${ad_{X_{k + 1}}}$ restricted on ${\cal G}^2$  is automorphic, there exists 
$Y\in {\cal G}^2$ such that $ad_{X_{k + 1}}\left(Y \right) = X$.

By changing $X'_1 = X_1 + Y$, we get  $\left[X'_1, X_{k + 1} \right] 
= C_{1,k + 1}^{k + 1}{X_{k + 1}}$. Using the Jacobi identity for $X'_1, X_{k + 1}$, and an arbitrary element $Z \in {\cal G}^2$ we obtain $ad_{X_1}ad_{X_{k+1}}-ad_{X_{k+1}}ad_{X_1}=\alpha ad_{X_{k+1}}$, where $\alpha$ is some real constant. Since $ad_{X_{k+1}}$ is automorphic on ${\cal G}^2$, $\alpha$ must be zero. Therefore, $C_{1,k + 1}^{k + 1} = 0$, i.e. $\left[X'_1, X_{k + 1} \right] = 0$. So, we can suppose that $\left[X_1,X_{k + 1} \right] = 0$.

By the same way, we can suppose  
$\left[X_2,X_{k + 1}\right] = \cdots = \left[X_k,X_{k + 1}\right] = 0.$
\end{proof}
%---------------------------2.1.4
\begin{lemma}
$\left[X_i,X_j \right] = C_{ij}^{k + 1}X_{k+1}$ for all pairs of indices i, j such that $1 \le i < j \le k$.
\end{lemma}
\begin{proof}
Consider an arbitrary pair of indices $i, j$ such that $1 \le i < j \le k$. Note that 
$\left[X_i, X_j \right] = \sum\limits_{l = k+1}^n {C_{ij}^l{X_l}} 
= C_{ij}^{k + 1}X_{k+1} + \sum\limits_{l = k+2}^n {C_{ij}^l{X_l}}$. 
By using the Jacobi identity, we have 
\begin{description}
   \item[]\qquad $\left[X_i,\left[X_j, X_{k + 1} \right] \right] + \left[X_j,\left[X_{k + 1}, X_i \right] \right] + \left[X_{k + 1},\left[X_i, X_j \right]\right] = 0$
   \item[]$\Longrightarrow\left[X_{k + 1},\left[X_i, X_j\right]\right] 
   = \left[X_{k + 1},\,C_{ij}^{k + 1}X_{k+1} + \sum\limits_{l = k+2}^n {C_{ij}^l{X_l}}\right]= 0$
   \item[]$\Longrightarrow ad_{X_{k+1}}{\left(\sum\limits_{l = k + 2}^n {C_{ij}^l{X_l}} \right)} = 0$
   \item[]$\Longrightarrow\sum\limits_{l = k + 2}^n {C_{ij}^l{X_l}}  = 0$
   (because $ad_{X_{k + 1}}$ is automorphic on ${\cal{G}}^2$)
   \item[]$\Longrightarrow\left[ {{X_i},{X_j}} \right] = C_{ij}^{k + 1}{X_{k + 1}}; 
   1 \le i < j \le k$.
\end{description}
\end{proof} 
We now prove Theorem 2.1.1. Namely, we will prove that if $\cal G$  is a real solvable Lie algebra such that ${\cal G}^{2}$  is non - trivial commutative and $dim{\cal G}^{2} = dim{\cal G}^{1}-1 = n-k-1,\,1\leq k<n$, then $\cal G$ is not an MD-algebra.
\vskip1cm
%--------------------------------------- prove Theorem 2.1.1
\begin{prove} \end{prove}
According to above lemmas, we can choose a suitable basis $\left(X_1, X_2,\cdots, X_n\right)$ of $\cal G$ which satisfies the following conditions: 
\begin{description}
   \item[]\hskip3cm $\left[X_i,X_j \right] = C_{ij}^{k + 1}X_{k + 1}, \,\,1 \le i < j
          \le k$;
   \item[]\hskip3cm $\left[X_i,X_{k+1} \right] = 0, \,\,1 \le i \le k$;
   \item[]\hskip3cm $\left[X_i,X_j \right] = \sum\limits_{l = k+2}^n
          {C_{ij}^l{X_l}}, \,\,1 \le i \le k+1,\,k+2 \le j \le n$.
\end{description}  

Moreover, the constructional constants $C_{ij}^{k + 1}$ can not concomitantly vanish and the matrix $A = {\left( C_{j, k+1}^l \right)}_{k+2 \le j,l \le n}$ is invertible because $ad_{X_{k+1}}$ restricted on ${\cal{G}}^2$ is automorphic.

Since A is invertible, there exist $\alpha_{k + 2},\cdots,\alpha_n \in {\mathbb {R}}$, which are not concomitantly vanished, such that 
$$A\left[\begin{array}{l}
\alpha_{k+2}\\ \,\,\,\,\vdots\\\alpha_n
\end{array}\right]
= \left[\begin{array}{l}1\\0\\ \vdots\\0
\end{array}\right] \in {\mathbb{R}}^{n - k - 1}.$$
Let $\left(X_1^*, X_2^*,\cdots, X_n^*\right)$ is the dual basis in ${\cal G}^*$ of 
$\left(X_1, X_2,\cdots, X_n\right)$. We choose   
${F_1} = X_{k+1}^*$ and ${F_2}= X_{k+1}^*+\alpha_{k+2}X_{k+2}^*
+\cdots+{\alpha_n}X_n^*$ in ${\cal G}^*$ .
It can easily be seen that $F_1, F_2$ are not perfectly vanishing
in ${\mathcal{G}}^{1}$. In the view of Proposition 1.2.3, if $\cal G$ is an MD-algebra then $\Omega_{F_1}, \Omega_{F_2}$ are orbits of maximal dimension, in particular we have 
$$rankB_{F_2} = dim \Omega_{F_1} = dim \Omega_{F_2} = rankB_{F_2}.$$ 

But it is easy to verify that $rankB_{F_2}\ge rankB_{F_1} + 2$. This contradiction proves that $\cal G$ is not an MD-algebra and the proof of Theorem 2.1.1 is therefore complete.\hfill $\square$
\vskip8mm
Now we consider an arbitrary real solvable Lie algebra $\cal{G}$ of dimension $n$ ($n \ge 5$) such that $dim{\cal G}^{1}= n - 1$. It is obvious that we can choose one basis $\left(X_1, X_2,\cdots, X_n\right)$ of $\cal G$ such that ${\cal G}^1 = gen\left(X_2, X_3,\cdots, X_n \right)$, $\,{\cal G}^2 \subset gen\left(X_3,\cdots, X_n\right)$ and ${\cal G}^2$ is commutative. Let $C_{ij}^l \left({1 \le i < j \le n}, 2 \le l \le n\right)$ are constructional constants of $\cal{G}$. Then the Lie brackets are given by the following formulas
$$\left[{{X_i},{X_j}} \right] = 
\sum\limits_{l = 2}^n {C_{ij}^l{X_l}} \left({1 \le i < j \le n} \right).$$

% --------------------------Theorem 2.1.5
\begin{theorem}
Let ${\cal G}$  be a real solvable Lie algebra of dimension $n$ such that its first derived ideal ${\cal G}^1$ is $(n-1)$-dimensional ($n \ge 5$) and its second derived ideal ${\cal G}^2$ is commutative. Then ${\cal G}$ is MDn-algebra if and only if $\,{\cal G}^1$  is commutative.
\end{theorem} 
%--------------- Some Lemmas for Theorem 2.1.5
In order to prove this theorem, once again, we also need some lemmas.
%--------------------------Lemma 2.1.6
\vskip8mm
\begin{lemma} If $\,\cal G$ is an  MD-algebra of dimension $n\,$ ($n \ge 5$) such that
$\,\dim{\cal G}^1 = n - 1$ then $\dim {\Omega_F}\in \lbrace 0, 2\rbrace \,$ for every $F \in {\cal G}^*$.
\end{lemma}
\begin{proof}
Let $ad_{X_1} = \left(a_{ij} \right)_{n - 1} \in End\left({\cal G}^1 \right)$. With $F_0 = X_2^*  \in {\cal G}^*$, the matrix of the bilinear form $B_{F_0}$ in the chosen basis as follows

$$B_{F_0} = \left[ \begin{array}{l}
0\,\,\,\,\,\,\, - {a_{12}}\,\,\,\,\,\, - {a_{13}}\,\,\,....\,\, - {a_{1n}}\\
{a_{12}}\,\,\,\,\,\,\,\,\,\,\,0\,\,\,\,\,\,\,\,\,\,\,\,\,\,\,\,0\,\,\,\,\,\,\,\,\,....\,\,\,\,\,\,\,\,0\\
{a_{13}}\,\,\,\,\,\,\,\,\,\,\,0\,\,\,\,\,\,\,\,\,\,\,\,\,\,\,\,0\,\,\,\,\,\,\,\,\,....\,\,\,\,\,\,\,\,0\\
....\,\,\,\,\,\,\,\,\,\,\,....\,\,\,\,\,\,\,\,\,\,\,....\,\,\,\,\,\,\,\,....\,\,\,\,\,\,....\\
{a_{1n}}\,\,\,\,\,\,\,\,\,\,0\,\,\,\,\,\,\,\,\,\,\,\,\,\,\,0\,\,\,\,\,\,\,\,\,\,....\,\,\,\,\,\,\,\,0
\end{array} \right].$$\\
It is plain that  $rank\,B_{F_0} = 2$. Since ${\cal G}$ is an MD-algebra, we get $\dim {\Omega_F}\in \lbrace 0, 2\rbrace \,$ for every $F \in {\cal G}^*$.
\end{proof} 

%----------------------Lemma 2.1.7
\begin{lemma}
Suppose that ${\cal G} = gen \left(X_1, X_2,\cdots, X_n \right)$  is a real solvable Lie algebra of dimension n such that\\ 
\centerline{${\cal G}^1 = gen \left(X_2, X_3,\cdots, X_n \right)$ and $\,{\cal G}^2 = gen \left(X_{k + 1},\cdots, X_n \right)$,\, $k > 1$.}
Let   
$$A = \left( {\begin{array}{*{20}{c}}
{C_{12}^2}& \ldots &{C_{1k}^2}\\
 \vdots & \ddots & \vdots \\
{C_{12}^k}& \cdots &{C_{1k}^k}
\end{array}} \right)$$\\
be the matrix established by the constructional constants 
${C_{1j}^l}$ $\left(2\leq j, l \leq k \right)$ of $\cal G$. Then A is invertible.
\end{lemma}
\begin{proof}
Since ${\cal G}^1 = \left[\cal G, \cal G \right]$, there exist real numbers $\alpha_{ij},\,1 \le i < j \le n$,  such that 
\begin{description}
   \item[]\hskip4mm ${X_2} =\sum\limits_{1 \le i < j \le n}{\alpha _{ij}} \left[ {{X_i},{X_j}} \right]$
   \item[]\hskip1cm $= \sum\limits_{j = 2}^k {{\alpha _{1j}}\left[ {{X_1},{X_j}} \right]}+\sum\limits_{j = k + 1}^n {{\alpha _{1j}}\left[ {{X_1},{X_j}} \right]}+\sum\limits_{2 \le i < j \le n} {{\alpha _{{ij}}}\left[ {{X_i},{X_j}} \right]}$
   \item[]\hskip1cm $=\sum\limits_{j = 2}^k {{\alpha _{1j}}\left[ {{X_1},{X_j}} \right]+ LC_1\left({{{\cal G}^2}} \right)}$
   \item[]\hskip1cm $=\sum\limits_{j = 2}^k {{\alpha _{1j}}{\left(\sum\limits_{l = 2}^n {C_{1j}^l{X_l}}\right)}+ LC_1\left({\cal G}^2 \right)}$
   \item[]\hskip1cm $=\sum\limits_{j = 2}^k {{\alpha _{1j}}{\left(\sum\limits_{l = 2}^k {C_{1j}^l{X_l}}+\sum\limits_{l = k+1}^n {C_{1j}^l{X_l}}\right)}+ LC_1\left({\cal G}^2 \right)}$
   \item[]\hskip1cm $=\sum\limits_{l = 2}^k {\sum\limits_{j = 2}^k {C_{1j}^l{\alpha_{1j}}X_l}+ LC_2\left({\cal G}^2 \right)},$
\end{description}
where $LC_1\left({\cal G}^2 \right),\,LC_2\left({\cal G}^2 \right)$ are linear combinations of some definite vectors from the chosen basis of  ${\cal G}^2.$
This implies that there exists columnar vector ${Y_2} \in {{\mathbb {R}}^{k - 1}}$  such that $$AY_2 = \left[ \begin{array}{l}
1\\
0\\
 \vdots \\
0
\end{array} \right] \in {\mathbb {R}}^{k - 1}.$$
Similarly, there exist columnar vectors $Y_3,\cdots,{Y_k} \in {{\mathbb {R}}^{k - 1}}$ such that $$AY_3 = \left[ \begin{array}{l}
0\\
1\\
0\\
 \vdots \\
0
\end{array} \right],\, \cdots, \, AY_k = \left[ \begin{array}{l}
0\\
0\\
 \vdots \\
0\\
1
\end{array} \right] \in {\mathbb {R}}^{k - 1}.$$\\
Thus, there is a real matrix $P$ such that $A.P = I$, where $I$ is the identity $(k-1)$-matrix. So $A$ is invertible and Lemma 2.1.7 is proved completely.
\end{proof}

%---------------------Prove Theorem 2.1.5
\subsection*{Proof of Theorem 2.1.5}
%--------(Only if)
Firstly, we shall prove that, if ${\cal G} = gen \left(X_1, X_2,\cdots, X_n \right)\,\left(n \ge 5 \right)$ such that ${\cal G}^1= gen \left(X_2, X_3,\cdots,X_n \right)$  is non-commutative, then ${\cal G}$ is not an MD-algebra. Let ${{\cal G}^2} = gen \left(X_{k + 1},\cdots, X_n \right),\, 2 \le k < n $.

We need consider some cases which contradict each other as follows.
\begin{itemize}
\item[1.]
    $k = 2$. Then, $\dim{\cal G}^2 = \dim{\cal G}^1 - 1$. According to Theorem 2.1.1, {\cal G} is not an MD-algebra.
\item[2.]
    $k = 3$. That means that ${\cal G}^2 = gen \left(X_4,\cdots, X_n \right)$. 
Assume that $\cal G$ is an MD-algebra. Remember that 
$$\left[X_1, X_2 \right] = \sum\limits_{l = 2}^n {C_{12}^l{X_l}},\, \left[X_1, X_3 \right] = \sum\limits_{l = 2}^n {C_{13}^l{X_l}},\, \left[X_i, X_j \right] = \sum\limits_{l = 4}^n {C_{ij}^l{X_l}}\begin{scriptsize}
{\footnotesize •}
\end{scriptsize},$$
for all $j > 4$ when $i = 1$, $j \ge 3$ when $i = 2$ and $j > 3$ when $i = 3$.
According  to Lemma 2.1.7, the matrix $P = \left[ {\begin{array}{*{20}{c}}
{C_{12}^2}&{C_{12}^3}\\{C_{13}^2}&{C_{13}^3}\end{array}} \right]$
is invertible.\\
Let $F = {\alpha_1}X_1^* + {\alpha_2}X_2^* + \cdots  + {\alpha_n}X_n^*$  be an arbitrary element of ${\cal G}^*$, where $\alpha_1, \alpha_2,\cdots, \alpha_n \in \mathbb{R}$. The matrix of the bilinear form  ${B_F}$  is $${B_F} = \left[ {\begin{array}{*{20}{c}}
0&{ - F\left( {\left[ {{X_1},{X_2}} \right]} \right)}& \cdots &{ - F\left( {\left[ {{X_1},{X_n}} \right]} \right)}\\
{F\left( {\left[ {{X_1},{X_2}} \right]} \right)}&0& \cdots &{ - F\left( {\left[ {{X_2},{X_n}} \right]} \right)}\\
 \cdots & \cdots & \cdots & \cdots \\
{F\left( {\left[ {{X_1},{X_n}} \right]} \right)}&{F\left( {\left[ {{X_2},{X_n}} \right]} \right)}& \cdots &0
\end{array}} \right].$$

Now we consider the $4$-submatrices of ${B_F}$ established by the elements which are on the rows and the columns of the same numbers 1, 2, 3, i ($i>3$). Because $\cal G$ is an MD-algebra, so according to Lemma 2.1.6, we get $rank(B_F) \in \lbrace 0, 2\rbrace $, this implies that the determinants of these $4$-submatrices are zero for any $F \in {{\cal G}^*}$. By direct computations, using the following obvious result of Linear Algebra: {\sl the determinant of any skew-symmetric real $4$-matrix $\left( a_{ij} \right)_4$ is equal to zero if and only if} ${a_{12}}.{a_{34}} - {a_{13}}.{a_{24}} + {a_{14}}.{a_{23}} = 0$, we get $C_{2i}^l = C_{3i}^l = 0,l \ge 4$.
This implies  $\left[X_2, X_i \right] = \left[X_3, X_i \right] = 0, i \ge 4$. Note that ${\cal G}^2$ is commutative. So we have
\begin{description}
   \item[] \hskip2cm ${\cal G}^2 = \left[{\cal G}^1,{\cal G}^1 \right] = 
           gen \left(X_4,\cdots, X_n \right)$ 
   \item[] \hskip2.5cm $= gen \left(\left[X_i, X_j \right];i,j \ge 2 \right)$
   \item[] \hskip2.5cm $= gen \left(\left[X_2, X_3 \right]\right)$. 
\end{description} 
Thus, $n - 3 = \dim {\cal G}^2 \le 1$, i.e. $n \le 4$. This contradicts the hypothesis $n \ge 5$. That means $\cal G$ is not an MD-algebra.
\item[3.]
     $k \ge 4$. By an argument analogous to that used above, we also prove that $\cal G$ is not an MD-algebra.
\end{itemize}
%----------------(If)

Conversely, assume that ${\cal G}$ is a real solvable Lie algebra of dimension $n$ such that its first derived ideal is $(n-1)$-dimensional and commutative, i.e. ${{\cal G}^1} \equiv {\mathbb {R}}.{X_2} \oplus {\mathbb {R}}.{X_3} \oplus ... \oplus {\mathbb {R}}.{X_n} \equiv {{\mathbb {R}}^{n - 1}}$. We need show that $\cal G$ is an MD-algebra.

Let $F = {\alpha_1}X_1^* + {\alpha_2}X_2^* + \cdots + {\alpha_n}X_n^* \equiv \left( \alpha_1, \alpha_2,\cdots,\alpha_n \right) \in {\mathbb{R}}^n$  be an arbitrary element from ${\cal G}^* \equiv {\mathbb {R}}^n$, where $\alpha_1, \alpha_2,\cdots,\alpha_n  \in {\mathbb {R}}$. By simple computation, we can see that the matrix  of the bilinear form ${B_F}$ is
 $${B_F} = \left[ {\begin{array}{*{20}{c}}
0&{ - F\left( {\left[ {{X_1},{X_2}} \right]} \right)}& \cdots &{ - F\left( {\left[ {{X_1},{X_n}} \right]} \right)}\\
{F\left( {\left[ {{X_1},{X_2}} \right]} \right)}&0& \cdots &0\\
 \cdots & \cdots & \cdots & \cdots \\
{F\left( {\left[ {{X_1},{X_n}} \right]} \right)}&0& \cdots &0
\end{array}} \right].$$
It is clear that $rankB_F \in \left\{ {0,2} \right\}$. Hence, ${\cal G}$  is an MDn-algebra and Theorem 2.1.5 is proved completely. \hfill{$\square$}

%--------------------- Section 2.2
\subsection{Classification of MD5-algebras having\\
non-commutative derived ideals}
The following theorem is the main result of the paper. It gives the classification up to an isomorphism of MD5-algebras having non-commutative derived ideals.
%--------------Theorem 2.2.1
\begin{theorem}
Let ${\cal G}$  be an MD5-algebra such that the first derived ideal ${\cal G}^1 = \left[{\cal G},{\cal G}\right]$ is non-commutative. Then the following assertions hold.
\begin{itemize}
\item[(i)]If ${\cal G}$ is decomposable, then ${\cal G} \cong {\cal H} \oplus {\cal K}$, where ${\cal H}$ and ${\cal K}$ are MD-algebras of dimensions which are no more than 4.
\item[(ii)]If ${\cal G}$  is indecomposable, then we can choose a suitable basis 
$(X_1, X_2, X_3,$\\ $X_4, X_5)$ of \,${\cal G}$ such that \,${\cal G}^1 = gen \left(X_3, X_4, X_5 \right),\,\left[X_3, X_4 \right] = X_5$; operators $ad_{X_1}, ad_{X_2}$ act on ${\cal G}^1$ as the following endomorphisms
$$ad_{X_1} = \left( {\begin{array}{*{20}{c}}
1&0&0\\
0&1&0\\
0&0&2
\end{array}}\right),\,\, ad_{X_2} = \left( {\begin{array}{*{20}{c}}
0&-1&0\\
1&0&0\\
0&0&0
\end{array}} \right)$$ 
and the other Lie brackets are trivial.
\end{itemize}
\end{theorem}

We need to prove some lemmas before we prove Theorem 2.2.1.
%-----------Lemma 2.2.2
\begin{lemma}
Let ${\cal G}$ be a real solvable Lie algebra. For any $Z \in {\cal G}$ we consider $ad_Z$ as an operator acted on ${\cal G}^1$. Then we have $Trace \left(ad_Z \right) = 0$ for all $Z \in {\cal G}^1$.
\end{lemma}

\begin{proof}
Using the Jacobi identity for $X,Y \in {\cal G}$ and an arbitrary element $Z \in {{\cal G}^1}$, we have $\left[ {X,\left[ {Y,Z} \right]} \right] + \left[ {Y,\left[ {Z,X} \right]} \right] + \left[ {Z,\left[ {X,Y} \right]} \right] = 0$. So, $a{d_X} \circ a{d_Y} - a{d_Y} \circ a{d_X} = a{d_{\left[ {X,Y} \right]}}$. This implies $Trace\left( {a{d_{\left[ {X,Y} \right]}}} \right) = 0$. Note that ${{\cal G}^1} = \left[ {{\cal G},{\cal G}} \right]$ and $a{d_Z}$ is a linear map. So we get $Trace\left( {a{d_Z}} \right) = 0$ for all $Z \in {{\cal G}^1}$.
\end{proof}  
%-----------------------Lemma 2.2.3
\begin{lemma}
If \,${\cal G}$ is a real solvable Lie algebra with $\dim{\cal G}^1 = 2$ then \,${{\cal G}^1}$ is commutative.
\end{lemma}
\begin{proof}
We choose a basis $(X, Y)$ of ${{\cal G}^1}$. Assume that $\left[X, Y \right] = aX + bY$. So we have $ad_X = \left( {\begin{array}{*{20}{c}}
0&a\\
0&b
\end{array}} \right), ad_Y = \left( {\begin{array}{*{20}{c}}
{ - a}&0\\
{ - b}&0
\end{array}} \right)\in End({\cal G}^1)$. According to Lemma 2.2.2, we get $a = b = 0$. Hence, ${{\cal G}^1}$ is commutative.
\end{proof}
Now we are ready to prove Theorem 2.2.1 - The main result of the paper.
\vskip1cm
%----------Proof of Theorem 2.2.1
\begin{prove} \end{prove} 

It is clear that assertion (i) of Theorem 2.2.1 holds obviously. We only
need to prove assertion (ii). Let ${\cal G}$ be an indecomposable MD5-algebra with the first derived ideal ${\cal G}^1 = \left[{\cal G}, {\cal G} \right]$ is non - commutative and the second derived ideal ${\cal G}^2 = \left[{\cal G}^1, {\cal G}^1\right]$ is commutative. According to Theorems 2.1.1, 2.1.5 and Lemma 2.2.3, the dimensions of ${\cal G}^1$ and ${\cal G}^2 = \left[{\cal G}^1, {\cal G}^1\right]$ must be 3 and 1, respectively.
We choose a basis $\left(X_1, X_2, X_3, X_4, X_5 \right)$ such that ${\cal G}^1 = gen \left(X_3, X_4, X_5 \right)$ and ${\cal G}^2 = gen \left(X_5 \right)$ with the Lie brackets are given by \[\begin{array}{l}
\left[X_1, X_2 \right] = {a_3}{X_3} + {a_4}{X_4} + {a_5}{X_5},\\ 
\left[X_1, X_3 \right] = {b_3}{X_3} + {b_4}{X_4} + {b_5}{X_5}, \\ 
\left[X_1, X_4 \right] = {c_3}{X_3} + {c_4}{X_4} + {c_5}{X_5}, \\
\left[X_1, X_5 \right] = {d_3}{X_3} + {d_4}{X_4} + {d_5}{X_5}, \\ 
\left[X_2, X_3 \right] = {e_3}{X_3} + {e_4}{X_4} + {e_5}{X_5}, \\ 
\left[X_2, X_4 \right] = {f_3}{X_3} + {f_4}{X_4} + {f_5}{X_5}, \\
\left[X_2, X_5 \right] = {k_3}{X_3} + {k_4}{X_4} + {k_5}{X_5}, \\
\left[X_3, X_4 \right] = {g_5}{X_5}, \left[X_3, X_5 \right] = {h_5}{X_5}, 
\left[X_4, X_5 \right] = {l_5}{X_5},
\end{array}\]
where $a_i, b_i, c_i, d_i, e_i, f_i, k_i \,(i = 3, 4, 5)$ and $g_5, h_5, l_5$ are the definite real numbers.

%---------------------------Some remarks

Now we give some useful remarks as follows.
\begin{itemize}
\item[{\bf a.}] According to Lemma 2.2.2, $Trace\left(ad_{X_3}\right) = Trace\left( ad_{X_4}\right) = 0$. That means ${h_5} = {l_5} = 0$. 
\item[{\bf b.}] $g_5 \ne 0$ because ${\cal G}^2 = gen (X_5)$. By changing ${X_3}$  with ${X'_3} = {\frac{1}{g_5}}X_3$, we get $\left[X'_3, X_4 \right] = X_5$. So, we can  suppose right from the start that $\left[X_3, X_4 \right] = X_5$, i.e. $g_5 = 1$.
\item[{\bf c.}] ${d_3} = {d_4} = {k_3} = {k_4} = 0$ because ${\cal G}^2 = \mathbb{R}.X_5$ is an ideal of ${\cal G}$. So, we get
$$\left[X_1, X_5 \right] = {d_5}{X_5},\,\left[X_2, X_5 \right] = {k_5}{X_5}.$$
If $k_5 \ne 0$, by changing $X'_2 = X_1 - {\frac{d_5}{k_5}}X_2$ we get
$\left[X'_2, X_5 \right] = 0$. So, we can always assume that $k_5 = 0$. 
\item[{\bf d.}] By changing ${X_1}$ with  ${X'_1} = {X_1} - {c_5}{X_3} + {b_5}{X_4}$  and ${X_2}$  with  ${X'_2} = {X_2} - {f_5}{X_3} + {e_5}{X_4}$, we get\\
\centerline {$\left[X'_1, X_3 \right] = {b_3}{X_3} + {b_4}{X_4}$,  
$\left[X'_1, X_4 \right] = {c_3}{X_3} + {c_4}{X_4}$,} 
\centerline {$\left[X'_2, X_3 \right] = {e_3}{X_3} + {e_4}{X_4}$, 
$\left[X'_2, X_4 \right] = {f_3}{X_3} + {f_4}{X_4}$.} 

Thus, we can  suppose right from the start that ${b_5} = {c_5} = {e_5} = {f_5} = 0$.
 
\end{itemize}
Using the Jacobi identity for triads $X_1, X_2, X_i \,(i = 3, 4, 5)$, we obtain
\[(I)\left\{ \begin{array}{l}
{a_3} = {a_4} = 0,\\
{e_4}{c_3} = {b_4}{f_3},\\
{e_3}{b_4} + {e_4}{c_4} = {b_3}{e_4} + {b_4}{f_4},\\
{f_3}{b_3} + {f_4}{c_3} = {c_3}{e_3} + {c_4}{f_3},\\
{b_3} + {c_4} = {d_5},\\
{e_3} + {f_4} = 0.
\end{array} \right.\]
%--------------------------------------------------------------------------------------
So we can reduce the Lie brackets as follows
 \[\begin{array}{l}
\left[ {{X_1};{X_2}} \right] = \,\,\,\,\,\,\,\,\,\,\,\,\,\,\,\,\,\,\,\,\,\,\,\,\,\,\,\,\,\,\,\,\,{a_5}{X_5},\\
\left[ {{X_1};{X_3}} \right] = {b_3}{X_3} + {b_4}{X_4},\\
\left[ {{X_1};{X_4}} \right] = {c_3}{X_3} + {c_4}{X_4},\\
\left[ {{X_1};{X_5}} \right] = \,\,\,\,\,\,\,\,\,\,\,\,\,\,\,\,\,\,\,\,\,\,\,\,\,\,\,\left( {{b_3} + {c_4}} \right){X_5},\\
\left[ {{X_2};{X_3}} \right] = {e_3}{X_3} + \,{e_4}{X_4},\\
\left[ {{X_2};{X_4}} \right] = {f_3}{X_3} - {e_3}{X_4},\\
\left[ {{X_3};{X_4}} \right] = \,\,\,\,\,\,\,\,\,\,\,\,\,\,\,\,\,\,\,\,\,\,\,\,\,\,\,\,\,\,\,\,\,{X_5}.
\end{array}\]
Thus, Relations $(I)$ can be rewritten as follows
\[(II)\left\{ \begin{array}{l}
{e_4}{c_3} = {b_4}{f_3},\\
2{e_3}{b_4} = {e_4}\left( {{b_3} - {c_4}} \right),\\
2{c_3}{e_3} = {f_3}\left( {{b_3} - {c_4}} \right).
\end{array} \right.\]
%--------------------- Some cases

Now we need consider the following cases which contradict each other.
%----------------------Case 1
\begin{itemize}
   \item[{\bf Case 1:}] ${e_3} = {e_4} = 0.$

\[(II) \Leftrightarrow \left\{\begin{array}{l}
{b_4}{f_3} = 0,\\
{f_3}\left( {{b_3} - {c_4}} \right) = 0.
\end{array} \right.\]
   \item[{\bf 1.1.}] Assume that ${f_3} = 0$. Then, Relations $(II)$ is automatically satisfied.\\
   \item[] By choosing $F_1 = X_3^*\in {\cal G}^*$, we get $rank{B_{{F_1}}} = 2$. 
   \item[] Now we choose $F_2 = X_5^* \in {\cal G}^*$. By simple computation, we obtain
\[{B_{{F_2}}} = \left[ {\begin{array}{*{20}{c}}
0&{ - {a_5}}&0&0\\
{{a_5}}&0&0&0\\
0&0&0&{ - 1}\\
{\begin{array}{*{20}{c}}
0\\
{{b_3} + {c_4}}
\end{array}}&{\begin{array}{*{20}{c}}
0\\
0
\end{array}}&{\begin{array}{*{20}{c}}
1\\
0
\end{array}}&{\begin{array}{*{20}{c}}
0\\
0
\end{array}}
\end{array}\,\,\,\,\,\,\,\begin{array}{*{20}{c}}
{ - \left( {{b_3} + {c_4}} \right)}\\
0\\
0\\
{\begin{array}{*{20}{c}}
0\\
0
\end{array}}
\end{array}} \right]\]
Since ${\cal G}$ is an MD-algebra, this implies that $rank{B_{{F_2}}} = rank{B_{{F_1}}} = 2$. That fact implies that ${a_5} = {b_3} + {c_4} = 0$, so ${\cal G}$ is decomposable, which is a contradiction. Thus, this case cannot happen.\\

   \item[{\bf 1.2.}] Now we assume that ${f_3} \ne 0$. Then $b_4 = 0, b_3 = c_4$. \\
By changing $X_1$ and $X_2$ with ${X'_1} = {X_1} - {c_3}{X'_2}$ and 
$X'_2 = {\frac{1}{f_3}X_2}$ we can suppose right from the start that $f_3 = 1, c_3 = 0$. Because of the dimension of ${\cal G}^1$ is 3, $b_3 \ne 0$. By changing $X_1$  with $X'_1 = {\frac{1}{b_3}}{X_1}$, we can always assume that $b_3 = 1$. By changing  $X_2$ with $X'_2 = X_2 - \frac{a_5}{2}{X_5}$, we can assume that $a_5 = 0$. Now, we choose $F_3 = X_3^*$ and $F_4 = X_4^*$ from ${\cal G}^*$, we get $rank{B_{{F_3}}} = 2,\,rank{B_{{F_4}}} = 4$. This cannot happen because ${\cal G}$ is an MD-algebra.

%--------------------Case 2
\item[{\bf Case 2:}]  ${e_4} = 0,{e_3} \ne 0$
\[\left( {II} \right) \Leftrightarrow \left\{ \begin{array}{l}
{b_4}{f_3} = 0,\\
{b_4}{e_3} = 0,\\
2{c_3}{e_3} = {f_3}\left( {{b_3} - {c_4}} \right).
\end{array} \right. \Leftrightarrow \left\{ \begin{array}{l}
{b_4} = 0,\\
2{c_3}{e_3} = {f_3}\left( {{b_3} - {c_4}} \right).
\end{array} \right.\]
By changing $X_2$  with $X'_2 = {\frac{1}{e_3}}{X_2}$, we can assume that $e_3 = 1$. Now Relations $(II)$ can be rewritten as follows
 \[(III)\left\{ \begin{array}{l}
{b_4} = 0,\\
2{c_3} = {f_3}\left( {{b_3} - {c_4}} \right).
\end{array} \right.\]
By changing $X_4$ with $X'_4 = X_4 - \frac{f_3}{2}{X_3}$  and  
$X_1$ with $X'_1 = X_1 - {b_3}{X_2}$, we can suppose that  $b_3 = f_3 = 0$.
From Relations (III) we get  $b_4 = c_3 = 0$.
Let $F = \alpha \,X_1^* + \beta X_2^* + \gamma X_3^* + \delta X_4^* + \sigma X_5^* \equiv \left( {\alpha ,\beta ,\gamma ,\delta ,\sigma } \right)$  be an arbitrary element from ${{\cal G}^*} \equiv {\mathbb {R}^5}$; $\alpha ,\beta ,\gamma ,\delta,\sigma  \in \mathbb {R}$. By simple computation, we obtain the matrix of the bilinear form ${B_F}$ as follows
 \[{B_F} = \left[ {\begin{array}{*{20}{c}}
0&{ - {a_5}\sigma }&0&{ - {c_4}\delta }\\
{{a_5}\sigma }&0&{ - \gamma }&\delta \\
0&\gamma &0&{ - \sigma }\\
{\begin{array}{*{20}{c}}
{{c_4}\delta }\\
{{c_4}\sigma }
\end{array}}&{\begin{array}{*{20}{c}}
\delta \\
0
\end{array}}&{\begin{array}{*{20}{c}}
\sigma \\
0
\end{array}}&{\begin{array}{*{20}{c}}
0\\
0
\end{array}}
\end{array}\,\,\,\,\,\,\,\begin{array}{*{20}{c}}
{ - {c_4}\sigma }\\
0\\
0\\
{\begin{array}{*{20}{c}}
0\\
0
\end{array}}
\end{array}} \right]\]
Since ${\cal G}$ is an MD-algebra, $rank{B_F}$  only get two values zero or two. Hence, it is easy to prove that  $c_4 = a_5 = 0$. But this implies that ${\cal G}$ is decomposable. This is a contradiction. Thus, this case cannot happen.

%-------------------case 3
\item[{\bf Case 3:}]${e_4} \ne 0$. 

By changing $X_4$  with $X'_4 = {e_4}{X_4} + {e_3}{X_3}$, we can assume that 
$e_3 = 0, e_4 = 1$. Now, Relations $(II$) can be rewritten as follows
$$\left\{ \begin{array}{l}
{c_3} = {b_4}{f_3},\\
{b_3} = {c_4}.
\end{array} \right.$$
By changing  $X_1$ with $X'_1 = X_1 - {b_4}{X_2}$, we can assume that $b_4 = 0$. Putting this into the relation above, we get $c_3 = 0$. 

Then, the Lie brackets in $\cal G$ can be reduced as follows
\[\begin{array}{l}
\left[X_1, X_2 \right] = \,\,\,\,\,\,\,\,\,\,\,\,\,\,\,\,\,\,\,\,\,\lambda {X_5},\\
\left[X_1, X_3 \right] = \mu {X_3},\\
\left[X_1, X_4 \right] = \,\,\,\,\,\,\,\,\,\mu {X_4},\\
\left[X_1, X_5 \right] = \,\,\,\,\,\,\,\,\,\,\,\,\,\,\,\,\,\,\,\,\,2\mu {X_5},\\
\left[X_2, X_3 \right] = \,\,\,\,\,\,\,\,\,\,\,\,\,{X_4},\\
\left[X_2, X_4 \right] = \,\theta {X_3},\\
\left[X_3, X_4 \right] = \,\,\,\,\,\,\,\,\,\,\,\,\,\,\,\,\,\,\,\,\,\,\,\,{X_5}.
\end{array}\]

Let $F = \alpha \,X_1^* + \beta X_2^* + \gamma X_3^* + \delta X_4^* + \sigma X_5^* \in {\cal G}^*$  be an arbitrary element from ${\cal G}^* \equiv {\mathbb {R}^5}$. Then by simple computation, we see that
\[{B_F} = \left[ \begin{array}{l}
0\,\,\,\,\, - \lambda \sigma \,\,\,\, - \mu \gamma \,\,\,\, - \mu \delta \,\,\,\, - 2\mu \sigma \\
\lambda \sigma \,\,\,\,\,\,\,\,\,0\,\,\,\,\,\,\,\,\, - \delta \,\,\,\,\,\, - \theta \gamma \,\,\,\,\,\,\,\,\,\,\,\,0\\
\mu \gamma \,\,\,\,\,\,\,\,\delta \,\,\,\,\,\,\,\,\,\,\,\,\,\,0\,\,\,\,\,\,\,\, \,- \sigma \,\,\,\,\,\,\,\,\,\,\,\,\,\,\,0\\
\mu \delta \,\,\,\,\,\,\,\,\theta \gamma \,\,\,\,\,\,\,\,\,\,\,\,\sigma \,\,\,\,\,\,\,\,\,\,\,\,\,\,0\,\,\,\,\,\,\,\,\,\,\,\,\,\,\,\,\,0\\
2\mu \sigma \,\,\,\,0\,\,\,\,\,\,\,\,\,\,\,\,\,\,\,0\,\,\,\,\,\,\,\,\,\,\,\,\,\,\,0\,\,\,\,\,\,\,\,\,\,\,\,\,\,\,\,\,0
\end{array} \right].\]

%-----------------------------------------CASE 3.1
   \item[{\bf 3.1.}] Assume $\mu  \ne 0$.
   \item[] When $\sigma  \ne 0$, we have $rank{B_F} = 4$. Because ${\cal G}$ is an
    MD5-algebra, we get $rank{B_F} \in \lbrace 0,4 \rbrace $ for all $\alpha ,\beta ,
   \delta ,\gamma,\sigma \in \mathbb {R}$. But this only can happen if $\theta  < 0$.
By changing $X_1, X_2, X_4, X_5$  with $X'_1 = \mu {X_1}, X'_2 = \sqrt{- \theta} {X_2}, X'_4 = \sqrt{- \theta}{X_4}, X'_5 = \sqrt{- \theta}{X_5}$, we can assume that $\theta = - 1$. By changing $X_2$  with $X'_2 = X_2 - \frac{\lambda}{2}{X_5}$, we can assume that $\lambda  = 0$. Hence, the Lie brackets of $\cal G$ can be reduced as follows
\[\begin{array}{l}
\left[ {{X_1};{X_3}} \right] = {X_3},\\
\left[ {{X_1};{X_4}} \right] = \,\,\,\,\,\,\,\,\,{X_4},\\
\left[ {{X_1};{X_5}} \right] = \,\,\,\,\,\,\,\,\,\,\,\,\,\,\,\,\,\,\,\,\,2{X_5},\\
\left[ {{X_2};{X_3}} \right] = \,\,\,\,\,\,\,\,\,\,\,\,{X_4},\\
\left[ {{X_2};{X_4}} \right] = \, - {X_3},\\
\left[ {{X_3};{X_4}} \right] = \,\,\,\,\,\,\,\,\,\,\,\,\,\,\,\,\,\,\,\,\,\,\,\,{X_5}.
\end{array}\]
%------------------------case 3.2--------------------------
   \item[{\bf 3.2}] Assume $\mu  = 0$. By the same way, we consider the matrix of the bilinear form ${B_F}$ and obtain $\lambda  = 0$. But this shows that ${\cal G}$ is decomposable. This construction show that this case can not happen. 
\end{itemize}
 
The theorem 2.2.1 is proved completely. So there is only one MD5-algebra having non-commutative derived ideal.$\hfill \square$

\subsection*{CONCLUDING REMARK}
Let us recall that each real Lie algebra $\cal G$ defines only one connected and simply connected Lie group G such that Lie(G) = $\cal G$. Therefore we obtain only one connected and simply connected MD5-group corresponding to the MD5-algebra given in Theorem 2.2.1. In the next paper, we shall describe the geometry of K-orbits of considered MD5-group, describe topological properties of MD5-foliation formed by the generic K-orbits of this MD5-group and give the characterization of the Connes's  $C^*$-algebra associated to this MD5-foliation.
%----------- References

\end{document}